\documentclass[11pt,a4paper]{amsart}

\usepackage{mathrsfs}
\usepackage{bbm}

\usepackage{amsmath,amssymb,amsthm}

\theoremstyle{plain}
\newtheorem{theorem}{Theorem}[section]
\newtheorem{proposition}[theorem]{Proposition}

\newtheorem{corollary}[theorem]{Corollary}

\theoremstyle{definition}

\newtheorem{remark}[theorem]{Remark}

\makeatletter

\newcommand{\IR}{\mathbbm{R}}

\newcommand{\IX}{\mathbbm{X}}

\def\%#1{\mathcal{#1}}
\newcommand{\law}{\mathscr{L}}

\newcommand{\nsig}{\Sigma\kern-0.5em\raise0.2ex\hbox to 0pt{$\mid$}\kern0.5em}

\newcommand{\toinf}{\to\infty}

\newcommand{\ahalf}{{\textstyle\frac{1}{2}}}

\newcommand{\eq}{\eqref}
\newcommand{\IE}{\mathbbm{E}}
\newcommand{\IP}{\mathbbm{P}}
\newcommand{\Var}{\mathop{\mathrm{Var}}}
\newcommand{\Cov}{\mathop{\mathrm{Cov}}}

\newcommand{\Id}{\mathop{\mathrm{Id}}}

\newcommand{\MVN}{\mathop{\mathrm{MVN}}}
\def\be#1\ee{\begin{equation*}#1\end{equation*}}
\def\ben#1\ee{\begin{equation}#1\end{equation}}

\def\bes#1\ee{\begin{equation*}\begin{split}#1\end{split}\end{equation*}}
\def\besn#1\ee{\begin{equation}\begin{split}#1\end{split}\end{equation}}

\def\bg#1\ee{\begin{gather*}#1\end{gather*}}
\def\bgn#1\ee{\begin{gather}#1\end{gather}}

\def\bm#1\ee{\begin{multline*}#1\end{multline*}}
\def\bmn#1\ee{\begin{multline}#1\end{multline}}

\def\ba#1\ee{\begin{align*}#1\end{align*}}
\def\ban#1\ee{\begin{align}#1\end{align}}

\def\bklr#1{\bigl(#1\bigr)}
\def\bbklr#1{\Bigl(#1\Bigr)}
\def\bbbklr#1{\biggl(#1\biggr)}

\def\bbbklrl{\biggl(}

\def\bbbklrr{\biggr)}

\def\bklg#1{\bigl\{#1\bigr\}}

\def\bbbklg#1{\biggl\{#1\biggr\}}

\def\bbbklgl{\biggl\{}
\def\bbbbklgl{\Biggl\{}

\def\bbbklgr{\biggr\}}
\def\bbbbklgr{\Biggr\}}

\def\norm#1{\Vert#1\Vert}
\def\bnorm#1{\bigl\Vert#1\bigr\Vert}

\def\abs#1{\vert#1\vert}
\def\babs#1{\bigl\vert#1\bigr\vert}
\def\bbabs#1{\Bigl\vert#1\Bigr\vert}
\def\bbbabs#1{\biggl\vert#1\biggr\vert}

\def\mid{\vert}
\def\bmid{\bigm\vert}

\def\^#1{\ifmmode {\mathaccent"705E #1} \else {\accent94 #1} \fi}
\def\~#1{\ifmmode {\mathaccent"707E #1} \else {\accent"7E #1} \fi}
\def\*#1{#1^\ast}
\def\>#1{\vec{#1}}
\def\.#1{\dot{#1}}

\def\leq{\leqslant}
\def\geq{\geqslant}

\def\atop{\@@atop}

\newcount\minute
\newcount\hour
\newcount\hourMins
\def\zeroPadTwo#1{\ifnum #1<10 0\fi#1}
\def\now{%
\minute=\time
\hour=\time \divide \hour by 60
\hourMins=\hour \multiply\hourMins by 60
\advance\minute by -\hourMins
\zeroPadTwo{\the\hour}:\zeroPadTwo{\the\minute}%
}

\def\proofsection#1{%
\par\bigskip%
\noindent{\it #1}%
\nopagebreak\medskip
\@afterheading\@afterindentfalse\par}

\makeatother

\def\cal{\mathcal}

\newcommand{\bX}{{{W}}}
\newcommand{\tI}{{\tilde{I}}}

\newcommand{\RR}{\IR} 
\def\bea#1\ena{\begin{eqnarray}#1\end{eqnarray}}
\def\beas#1\enas{\begin{eqnarray*}#1\end{eqnarray*}}
\def\beq#1\enq{\begin{equation}#1\end{equation}}

\newcommand{\ignore}[1]{}

\numberwithin{equation}{section}

\begin{document}

\title[Embedding for multivariate normal approximation]{U-statistics and random
subgraph counts: Multivariate normal approximation
via  exchangeable pairs and embedding}
\author{Gesine~Reinert \and Adrian~R\"ollin}
\maketitle
\vspace{-0.4cm}

\begin{center}
\it University of Oxford and National University of Singapore
\end{center}



\begin{abstract} 
In a recent paper \cite{Reinert2009} a new approach---called the
``embedding method''---was introduced, which allows to make use of exchangeable
pairs for normal
and multivariate normal approximation with Stein's method in cases where the
corresponding couplings do not satisfy a certain linearity condition. The key
idea is to embed the problem into a higher dimensional space in such a way that
the linearity condition is then satisfied. Here we apply the embedding to
U-statistics as well as to subgraph counts in random graphs. 
\end{abstract}


\section{Introduction} 

Stein's method, first introduced in the 70s \cite{Stein1972}, has proven a
powerful tool for assessing distributional distances, such as to the normal
distribution, in the presence of dependence. When considering sums $W$ of random
variables, the dependence between these random variables needs to be weak in
order for the distance to a normal distribution to be small. For quantifying
weak dependence, Stein \cite{Stein1986} introduced the
method of exchangeable pairs: construct a sum $W'$ such that $(W,W')$ form an
exchangeable pair, and such that $\IE^W(W'-W)$ is (at least approximately)
linear in $W$. This linearity condition arises naturally when thinking of
correlated bivariate normals. The generalisation of this approach to a
multivariate setting remained untackled until recently Chatterjee and
Meckes~\cite{Chatterjee2007} solved the problem in the case of exchangeable
vectors $(W,W')$ such that $\IE^W(W'-W) = - \lambda W + R$ for a scalar
$\lambda$ and a remainder vector $R$ such that $\IE | R | $ is small. This is a
rather special case;  the authors \cite{Reinert2009} tackled the general setting
that 
\ben                                                                \label{1}
    \IE^W(W'-W) = - \Lambda W + R
\ee
for a matrix $\Lambda$ and a vector $R$ with small $\IE|R|$  is treated.
In in a followup paper by Meckes~\cite{Meckes2009} the results by
Chatterjee and Meckes \cite{Chatterjee2007} and by the authors
\cite{Reinert2009}
are combined using slightly different smoothness conditions on test functions as
compared to \cite{Reinert2009}; non-smooth test functions are not treated by
Meckes \cite{Meckes2009}, but the bounds obtained there improve
on those from \cite{Reinert2009} for the example of $d$-runs with respect
to smooth test functions. 

A surprising finding in \cite{Reinert2009} was that it is often
possible to embed
a random vector $W$ into a  random vector $\hat{W}$ of larger, but still finite,
dimension, such that \eq{1} holds with $R=0$; yet this embedding does not correspond to Hoeffding projections. Here we explore the embedding
method further, by illustrating its use on two important examples. The first
example is complete non-degenerate U-statistics, and the second example
considers the joint count of edges and triangles in Bernoulli random graphs. In
both examples the limiting covariance matrix is not of full rank; yet the bounds
on the normal approximation are of the expected order.  

The paper is organised as follows. In Section \ref{sec1} we review the
theoretical results in \cite{Reinert2009}, giving bounds on the
distance to
normal under the linearity condition \eq{1}, both for smooth test functions and
for non-smooth test functions. In Section \ref{sec4} we discuss the embedding
method, and point out a link to Rademacher integrals and chaos decompositions. Section~\ref{secu}
illustrates the embedding method for complete non-degenerate  U-statistics; the
embedding vector contains lower-order U-statistics which are obtained via fixing
components. Section~\ref{randomgraphs} gives a normal approximation for the
joint counts of the number of edges and the number of triangles in a Bernoulli
random graph; to our knowledge these are the first explicit bounds for this
multivariate problem. The embedding method suggests to count the number of
2-stars as well, which makes the results not only more informative but also, surprisingly, 
 easier to derive.   

\section{Theoretical bounds for a multivariate normal approximation} 
\label{sec1}
\subsection{Notation}

Denote by $W = (W_1,W_2,\dots,W_d)^t$ random vectors in $\IR^d$, where $W_i$ are
$\IR$-values random variables for $i=1,\dots,d$. We denote by $\Sigma$
symmetric, non-negative definite matrices, and hence by $\Sigma^{1/2}$ the
unique symmetric square root of~$\Sigma$. Denote by $\Id$ the identity matrix,
where we omit the dimension~$d$. Throughout this article, $Z$ denotes a random
variable having standard $d$-dimensional multivariate normal distribution. We
abbreviate the transpose of the inverse of a matrix $\Lambda$ as $\Lambda^{-t} :=
(\Lambda^{-1})^t$.

For derivatives of smooth functions $h:\IR^d\to\IR$, we use the notation $\nabla
$ for the gradient operator. 
Denote by $\norm{\cdot}$ the supremum norm for both functions and  matrices. If
the corresponding derivatives exist for some function $h:\IR^d\to\IR$, we
abbreviate $\abs{h}_1 :=\sup_{i}\bnorm{\frac{\partial}{\partial x_i} h}$,
$\abs{h}_2 :=\sup_{i,j}\bnorm{\frac{\partial^2}{\partial x_i \partial x_j}h}$,
and so on.



We start by considering smooth test functions.

\begin{theorem}[c.f. Theoem~2.1 \cite{Reinert2009}] \label{thm1} Assume that
$(W,W')$ is an exchangeable pair of\/ $\IR^d$-valued random variables such that 
\ben
    \IE W = 0,\qquad\IE W W^t = \Sigma, \label{2}
\ee 
with $\Sigma \in\IR^{d\times d}$ symmetric and positive definite. Suppose
further that~\eq{1} is satisfied for an invertible matrix $\Lambda$ and a
$\sigma(W)$-measurable random variable~$R$. Then, if $Z$ has $d$-dimensional
standard normal distribution, we have for every three times differentiable
function~$h$,
\ben                                                           \label{3}
    \babs{\IE h(W)-\IE h(\Sigma^{1/2}Z)}
    \leq 
    \frac{\abs{h}_2}{4}  A
    +\frac{\abs{h}_3}{12} B    
    +\bbklr{\abs{h}_1 +\ahalf d\norm{\Sigma}^{1/2} \abs{h}_2} C   ,  
\ee
where, with $\lambda^{(i)} = \sum_{m=1}^d\abs{(\Lambda^{-1})_{m,i}}$, 
\ba
A &= \sum_{i,j=1}^d{\lambda^{(i)}} \sqrt{\Var{\IE^W(W'_i-W_i)(W'_j-W_j)}},\\
B &= \sum_{i,j,k=1}^d{\lambda^{(i)}} \IE\babs{(W'_i-W_i)(W'_j-W_j)(W'_k-W_k)},\\
C &= \sum_{i}\lambda^{(i)}\sqrt{\IE R_i^2}.
\ee
\end{theorem}

The proof of Theorem~\ref{thm1} is based on the Stein characterization of the
normal distribution that  ${Y} \in \IR^d$ is a multivariate normal
$\MVN(0,\Sigma)$ if and only if
\ben                                                                \label{4} 
    \IE Y^t \nabla f(Y)  
    = \IE \nabla^t \Sigma \nabla f(Y), 
    \quad \text{for all smooth $f:{\IR}^d \rightarrow {\IR}$}. 
\ee

In the paper by Meckes \cite{Meckes2009} a different norm for functions and for
operators is used, to obtain a similar result, and the difference in the bounds depending on the chosen norm 
is illustrated for the
example of runs on the line. 

Theorem~\ref{thm1} can be extended to allow for covariance matrices which are
not full rank, using the triangle inequality in conjunction with the following
proposition.  

\begin{proposition}[c.f.~Proposition~2.9 \cite{Reinert2009}] \label{lem} Let
$X$ and $Y$ be $\IR^d$-valued normal variables with distributions $X \sim
\MVN(0, \Sigma)$ and $Y \sim \MVN(0, \Sigma_0)$, where $\Sigma =
(\sigma_{i,j})_{i,j=1, \ldots, d}$ has full rank, and $\Sigma_0=
(\sigma^0_{i,j})_{i,j=1, \ldots, d}$ is non-negative definite. Let $h:\IR^d\to
\IR$ have 3 bounded derivatives. Then
\be
    \babs{\IE h(X)-\IE h(Y)}
    \leq \frac{1}{2}  \abs{h}_2 \sum_{i, j=1}^d \abs{\sigma_{i,j} -
\sigma^0_{i,j}}. 
\ee
\end{proposition}  


For non-smooth test functions, following Rinnot and Rotar \cite{Rinott1996}, let
$\Phi$ denote the standard normal distribution in $\RR^d$, and $\phi$ the
corresponding density function. For $h: \RR^d \rightarrow R$ set
\ba 
    h_\delta^+ (x) &= \sup\{h(x+y): \vert y \vert \leq\delta\}, \quad
    h_\delta^-(x)  = \inf\{h(x+y): \vert y \vert \leq\delta\},\\
   \mbox{ and }  {\tilde h} (x, \delta) &= h_\delta^+ (x) - h_\delta^-(x). 
\ee
Let $ {\cal H} $ be a class of measurable functions $\RR^d \rightarrow R$ which
are uniformly bounded by 1. Suppose that for any $h  \in{\cal H} $,  
 for any $\delta > 0$, $h_\delta^+ (x)$ and $h_\delta^-(x)$ are in ${\cal
H}$;  
 for any $d \times d $ matrix $A$ and any 
vector $b \in \RR^d$, $h(Ax+b) \in {\cal H}$; and 
for some constant $a = a( {\cal H}, \delta)$,
$
 \sup_{h \in{\cal H}} \left\{ \int_{\RR^d} {\tilde h} (x, \delta) \Phi (dx)
  \right\} \leq a \delta $. 
Obviously we may assume $a \ge
1$. The class of indicators of measurable convex sets is such a class where $a
\leq 2 \sqrt{d}$; see the paper by Bolthausen and G\"otze \cite{Bolthausen1993}.

Let $W $ have mean vector
$0$ and covariance matrix $\Sigma$. If $\Lambda$ and $R$ are such that
\eq{1} is satisfied for $W$, then $Y = \Sigma^{-1/2} W$ satisfies \eq{1} with
$    \hat\Lambda = \Sigma^{-1/2} \Lambda \Sigma^{1/2} \text{ and } R'
    = \Sigma^{-1/2} R.$ 
With
\be
  \hat\lambda^{(i)} = \sum_{m=1}^d \abs{ (\Sigma^{-1/2} \Lambda^{-1}
  \Sigma^{1/2} )_{m,i} }
\ee
as well as 
\ban                                                             
  A' & = \sum_{i,j} \^\lambda^{(i)} 
  \sqrt{\Var \IE^Y  
  \sum_{k, \ell} \Sigma^{-1/2}_{i,k} \Sigma^{-1/2}_{j,\ell}(W_k'
    - W_k)(W_\ell' - W_\ell) } , \\                               
  B' &=  \sum_{i,j,   k }  \hat\lambda^{(i)} 
  \IE\bbbabs{ 
  \sum_{ r,s,t} \Sigma^{-1/2}_{i,r} \Sigma^{-1/2}_{j,s } \Sigma^{-1/2}_{k,
    t}(W_r'-W_r)(W_s'- W_s)(W_t' - W_t)}
\ee 
and 
\be                                                               \%label{9} 
  C' = \sum_{i} \hat\lambda^{(i)}\sqrt{\IE \bbklr{  \sum_k \Sigma^{-1/2}_{i,k}
    {R_k}}^2  },
\ee
we have the following result \cite{Reinert2009}.

\begin{corollary}\label{cor1} 
Let $W$ be as in Theorem~\ref{thm1}. Then, for all $h \in {\cal H}$ with $\vert
h \vert \leq1$,  there exist $\gamma=\gamma(d)$ and $a > 1$ such that
\be 
   \sup_{h \in{\cal H}} \abs{ \IE h(W) - \IE h(Z) }  
    \leq  \gamma^2\bbklr{  -  D' \log (T') + 
    \frac{B'}{ 2\sqrt{T'} } + C'  + a \sqrt{T'} }, 
\ee
with 
\be 
T' = \frac{1}{a^2 } \bbbklr{D' +  \sqrt{\frac{aB'}{2} +D'^2 } }^2 
\qquad\text{ and }\qquad
D' =  \frac{A'}{2}  +  C'  d. 
\ee
\end{corollary} 

\begin{remark} \label{cor2} 
We can simplify the above bound further. Using Minkowski's inequality we have
that $
    \Var \sum_{i=1}^k X_i       \leq k^2\sup_i {\Var X_i}, $ 
and thus obtain the simple estimate 
\bes
&\Var \IE^Y  
\sum_{k, \ell} \Sigma^{-1/2}_{i,k} \Sigma^{-1/2}_{j,\ell}(W_k'
- W_k)(W_\ell' - W_\ell) \\
&\qquad\leq
    d^4 \norm{\Sigma^{-1/2}}^4\sup_{k,\ell} 
        \Var\IE^W\bklg{(W_k'-W_k)(W_\ell'-W_\ell)
}
\ee
and hence
\be                                                               
A' \leq  d^3 \norm{\Sigma^{-1/2}}^2  \sum_{i} \hat\lambda^{(i)}\sup_{k,\ell}
\sqrt{ \Var\IE^W\bklg{(W_k'-W_k)(W_\ell'-W_\ell)}  };                           
\ee 
in $B'$ and $C'$ we could similarly bound   $\Sigma^{-1/2}_{i,k}$ by
$\norm{\Sigma^{-1/2}}$ to obtain
a simpler bound.    There are however examples, such as the random graph example
in Section~\ref{randomgraphs},
where $\norm{\Sigma^{-1/2}}$ provides a non-informative bound. 
\end{remark}

\begin{remark} Note that, if $(W,W')$ is exchangeable and \eq{1} is satisfied we
have
\besn                                                           \label{11}
    \IE (W'-W)(W'-W)^t 
    &= 2\IE W(\Lambda W)^t = 2\Sigma\Lambda^t
\ee
On the other hand, if we only have $\law(W)=\law(W')$, we obtain
\ben                                                            \label{12}
    \IE (W'-W)(W'-W)^t = \Lambda \Sigma + \Sigma\Lambda^t.
\ee
Hence, to check in an application whether the often tedious calculation of
$\Sigma$
and $\Lambda$ has been carried out correctly, we can combine Equations \eq{11}
and~\eq{12}, to conclude that, under the conditions of
Theorem~\ref{thm1}, we must have $\Lambda\Sigma = \Sigma\Lambda^t$.
\end{remark}



\section{The embedding method}
\label{sec4}

Assume that an $\ell$-dimensional random vector $W_{(\ell)}$ of
 interest is given. Often, the construction of an exchangeable pair
 $(W_{(\ell)},W_{(\ell)}')$ is
 straightforward. If, say, $W_{(\ell)}=W_{(\ell)}(\IX)$ is a function of i.i.d.\
random
 variables
 $\IX=(X_1,\dots,X_n)$, one can choose uniformly an index $I$ from $1$ to $n$, 
 replace $X_I$ by an independent copy $X'_I$, and define
$W_{(\ell)}':=W_{(\ell)}(\IX')$,
 where
 $\IX'$ is now the vector $\IX$ but with $X_I$ replaced by $X'_I$.
 
 In general there is  no hope that $(W_{(\ell)}, W_{(\ell)}')$ will satisfy
 Condition~\eq{1} with $R$ being of the required smaller order or even equal
 to zero, so that in this case Theorem~\ref{thm1} would not yield useful
 bounds.

 Surprisingly often it is possible, though, to extend $W_{(\ell)}$ to a vector
$W \in \IR^d$  such that we can construct an exchangeable pair $(W,W')$ which
satisfies Condition~\eq{1} with $R=0$. If we can bound the distance of the
distribution ${\cal L}(W)$ to a $d$-dimensional multivariate normal
distribution, then a bound on the distance  of the distribution ${\cal
L}(W_{(\ell)})$ to an $\ell-$dimentional  multivariate normal distribution
follows immediately. 

In
order to obtain useful bounds in Theorem \ref{thm1}, the embedding dimension $d$
should not be too large. In the examples below it will be obvious how to choose
$ W^{(d -\ell)}$ to make the construction work.

As a first illustration of the method, it was observed in 
\cite{Nourdin2009b} that for functions which depend on the first $d$ coordinates
of an infinite Rademacher sequence, that is, a sequence of symmetric $\{-1, 1\}$
random variables, the 
natural embedding vector is a vector of Rademacher integrals  of lower order. A
similar construction works fairly generally, as follows. 
Assume that $F = F(X_1,...,X_d)$ is a random variable that depends uniquely on
the first $d$ coordinates
 of a  sequence $X$ of i.i.d. mean zero random variables, with $E(F)=0$ and
$E(F^2)=1$, of the form 
\begin{eqnarray} \label{fdec}
F & = &  \sum_{n=1}^d \sum_{1\leq i_1<...<i_n\leq d }   n! f_n (i_1 , ..., i_n)
X_{i_1}\cdot \cdot \cdot X_{i_n} =:  \sum_{n=1}^d J_n(f_n);
\end{eqnarray}
such representations occur as chaotic decompositions for functionals of
Rade\-macher sequences. 
A natural exchangeable pair construction is as follows. Pick an index $I$  so that $P(I=i) = \frac{1}{d}$ for $i=1, \ldots, d$, independently of
$X_1,...,X_d$,
and if $I=i$ replace $X_i$ by an independent copy $X_i^*$ in all sums in the
decomposition \eqref{fdec}
which involve $X_i$. Call the resulting expression $F'$, and the corresponding
sums $J_n'(f_n); n=1, \ldots, d$. 
Now choosing as
embedding vector $ {W} = (J_1(f_1), \ldots, J_d (f_d)) ,$ we check that for all
$n=1, \ldots, d$, 
\begin{eqnarray*}
\lefteqn{ E( J_n'(f_n)-J_n(f_n) | {W} ) } \notag
\\ &=& - \frac{1}{d} \sum_{i=1}^d \sum_{1\leq i_1<...<i_n\leq d } 
\mathbf{1}_{\{ i_1, \ldots, i_n\}}(i)\,\, n! f_n
(i_1 , ..., i_n) E( X_{i_1}\cdot \cdot \cdot X_{i_n} | {W}) \notag \\
&=&-\frac{n}d\,J_n(f_n).
\end{eqnarray*}
Thus, with $ {W'} = (J_1'(f_1), \ldots, J_d' (f_d)) ,$ the condition \eqref{1}
is satisfied, with $\Lambda= (\lambda_{i,j})_{1\leq i,j\leq d}$ being zero off
the diagonal and $\lambda_{n,n} = \frac{n}{d}$ for $n=1, \ldots, d$. Note that,
although diagonal, the diagonal entries of this $\Lambda$ are not equal. It is
not possible to correct this by simple coordinate-wise scaling of $W$ as this
will change $\Sigma$ only and leave $\Lambda$ unaffected; see 
also the discussion in  \cite[Section~5]{Reinert2009}. Hence, again, the
generality of $\eq{1}$ is essential here.

\section{Complete non-degenerate $U$-statistics} \label{secu} 

Using the exchangeable pairs coupling, Rinott and Rotar \cite{Rinott1997} proved
a univariate normal approximation theorem for non-degenerate weighted
$U$-statistics with symmetric weight function under fairly mild conditions on
the weights. Using the typical coupling, where uniformly a random variable $X_i$
is choosen and replaced by an independent copy, they show that \eq{1} is
satisfied for the one-dimensional case and a non-trivial remainder term, being 
Hoeffding projections of smaller order. It should not be difficult (but
nevertheless cumbersome) to generalise their result to the multivariate case,
where $d$ different $U$-statistics are regarded based on the same sample of
independent random variables, such that \eq{1} is satisfied with $\Lambda=I$ and
non-trivial remainder term, again of lower order; for multivariate
approximations of several $U$-statistics see also the book by
Lee~\cite{Lee1990}.
However, as we want to emphasize  the use of Theorem~\ref{thm1} for non-diagonal
$\Lambda$, we take a different approach.

Let $X_1,\dots,X_n$ be a sequence of i.i.d.\ random elements taking values in
a measure space $\%X$. Let $\psi$ be a measurable and symmetric function from
$\%X^d$ to $\IR$, and, for each $k=1,\dots, d$, let
\be 
    \psi_k(X_1,\dots,X_k) := \IE\bklr{\psi(X_1,\dots,X_d)\bmid X_1,\dots,X_k }.
\ee
Assume without loss of generality that $\IE\psi(X_1,\dots,X_d) = 0$. For any
subset $\alpha\subset\{1,\dots,n\}$ of size $k$ write
$\psi_k(\alpha):=\psi(X_{i_1},\dots,X_{i_k})$ where the $i_j$ are the elements
of $\alpha$. Define the statistics
\be
    U_k :=\sum_{\abs{\alpha}=k}\psi_k(\alpha),
\ee 
where
$\sum_{E(\alpha)}$ denotes summation over all subsets
$\alpha\subset\{1,\dots,n\}$ which satisfy the property~$E$. Then $U_d$
 coincides with the usual $U$-statistics with kernel~$\psi$. Assume that $U_d$
is non-degenerate, that is, $\IP[\psi_1(X_1) = 0]< 1$. Put 
\be
    W_k := n^{1/2}{n\choose k}^{-1} U_k.
\ee
It is well known that $\Var W_k \asymp 
1$ (e.g.\ \cite{Lee1990}). Note also that, as $n\toinf$, $\Sigma:=\IE(WW^t)$
will
converge to a
covariance matrix with all entries equal to $\Var \psi_1(X_1)$ and which is
thus of rank $1$, as we assume non-degeneracy and hence $U_1 =
\sum_{i=1}^n\psi_1(X_i)$ will dominate the behaviour of each~$U_k$.

Using Stein's method and the approach of decomposable random variables,
Rai\v{c}~\cite{Raic2004} proved rates of convergence for vectors of
$U$-statistics,
where the coordinates are assumed to be uncorrelated (but nevertheless based
upon the same sample $X_1,\dots,X_n$). The next theorem can be seen as a
complement to Rai\v{c}'s results, as in our case from above, a normalization is
not appropriate.

\begin{theorem}\label{thm2}
With the above notation, and if $\rho := \IE \psi(X_1,\dots,X_d)^4<\infty$
we have for every three times differentiable function $h$
\be 
    \abs{\IE h(W) - \IE h(\Sigma^{1/2} Z)} \leq n^{-\frac{1}{2}} 
   \left( 4\rho^{1/2}d^6\abs{h}_2 + \rho^{3/4}d^7\abs{h}_3 \right). 
\ee
\end{theorem}
\begin{proof}

Let $X_1',\dots,X_n'$ be independent copies of $X_1,\dots,X_n$. Define the
random variables $\psi'_{j,k}(\alpha)$ analogously to $\psi_k(\alpha)$
but based on the sequence $X_1,\dots,X_{j-1},X'_j,X_{j+1},\dots,X_n$. 
Define the coupling as in \cite{Rinott1997}, that is, pick uniformly an index
$J$ from $\{1,\dots,n\}$ and replace $X_J$ by $X'_J$, so that $U'_k
= \sum_{\abs{\alpha}=k}\psi'_{J,k}(\alpha)$; it is easy to see that $(U',U)$ is
exchangeable. Note now that, if $j\not\in\alpha$,
$\psi'_{j,k}(\alpha)=\psi_k(\alpha)$, and, with $X=(X_1,\dots,X_n)$, that
$\IE^X\psi'_{j,k}(\alpha) = \psi_{k-1}(\alpha\setminus\{j\})$ if $j\in\alpha$.
Thus
\besn                                                           \label{14}
    \IE^{X} (U'_k - U_k) 
    &= \frac{1}{n}
    \sum_{j=1}^n\sum_{\abs{\alpha}=k,\atop \alpha\ni j}
     \IE^X\bklg{\psi'_{j,k}(\alpha)-\psi_k(\alpha)}\\
    &=-\frac{k}{n}U_k + \frac{1}{n}
    \sum_{j=1}^n\sum_{\abs{\alpha}=k,\atop \alpha\ni j}
     \psi_{k-1}(\alpha\setminus\{j\})\\
    &=-\frac{k}{n}U_k + \frac{n-k+1}{n}
    \sum_{\abs{\beta}=k-1} \psi_{k-1}(\beta)\\
    & = -\frac{k}{n} U_k + \frac{n-k+1}{n} U_{k-1},
\ee
where the second last equality follows from the observation that 
\be 
    \sum_{\abs{\alpha}=k,\atop\alpha\ni j}
   \psi_{k-1}(\alpha\setminus\{j\})
    = \sum_{\abs{\beta}=k-1,\atop\beta\not\ni j}
   \psi_{k-1}(\beta),
\ee
and thus, in the corresponding double sum of \eq{14}, every set $\beta$ of size
$k-1$
appears exactly $n-(k-1)$ times. Thus
\be
    \IE^{X} (W'_k - W_k) = -\frac{k}{n}(W_k - W_{k-1}).  
\ee
Hence, \eq{1} is satisfied for $R=0$ and 
\be
    \Lambda = \frac{1}{n}
    \begin{bmatrix}
    1&&& \\
    -2& 2\\
    &-3 & 3\\
    &&\ddots&\ddots\\
    &&&- d& d\\
    \end{bmatrix},
\ee
with lower triangular $\Lambda^{-1}$ such that, if $l\leq k$,
\be
    (\Lambda^{-1})_{k,l} = n / l,
\ee
thus, for $l=1,\dots,d$,
\ben                                                  \label{15}     
    \lambda^{(l)} \leq d n.
\ee

Define now $\eta_{j,k}(\alpha) := \psi'_{j,k}(\alpha)-\psi_{k}(\alpha)$. Then
we have for every $k,l=1,\dots,d$,
\besn                                                    \label{16}
    \IE^{X,X'}\bklg{(U'_k-U_k)(U'_l-U_l)} 
    &= \frac{1}{n}
    \sum_{j=1}^n\bbbklr{
    \sum_{\abs{\alpha}=k,\abs{\beta}=l,\atop \alpha\cap\beta\ni j}
    \eta_{j,k}(\alpha)\eta_{j,l}(\beta) 
    }
\ee
and
\besn                                                            \label{17}
    &\IE\bklr{\IE^{X,X'}\bklg{(U'_k-U_k)(U'_l-U_l)}}^2 \\
    &\qquad=\frac{1}{n^2}\sum_{i,j=1}^n
    \sum_{\abs{\alpha}=k,\abs{\beta}=l,\atop \alpha\cap\beta\ni i}
    \sum_{\abs{\gamma}=k,\abs{\delta}=l,\atop \gamma\cap\delta\ni j}
    \IE\bklg{\eta_{i,k}(\alpha)\eta_{i,l}(\beta)
    \eta_{j,k}(\gamma)\eta_{j,l}(\delta)  } .
\ee
Note now that, if the sets $\alpha\cup\beta$ and $\gamma\cup\delta$ are
disjoint (which can only happen if $i\neq j$),
\ben  
    \IE\bklg{\eta_{i,k}(\alpha)\eta_{i,k}(\beta)
    \eta_{j,l}(\gamma)\eta_{j,l}(\delta)  } = 
    \IE\bklg{\eta_{i,k}(\alpha)\eta_{i,k}(\beta)}
    \IE\bklg{\eta_{j,l}(\gamma)\eta_{j,l}(\delta) }
\ee
due to independence. The variance of \eq{16}, that is \eq{17} minus the
square of the expectation of \eq{16}, contains only summands where
$\alpha\cup\beta$ and $\gamma\cup\delta$ are not disjoint. Recall now  
that $\rho = \IE\psi(X_1,\dots,X_d)^4$. Bounding all the non-vanishing
terms simply by $32\rho$, it only remains to count the number of non-vanishing
terms.
Thus,
\bes
    &\Var\IE^{X,X'}(U'_k-U_k)(U'_l-U_l) \\
    &\qquad\leq \frac{1}{n^2}\sum_{i,j=1}^n
    \sum_{\abs{\alpha}=k,\abs{\beta}=l,\atop \alpha\cap\beta\ni i}
    \sum_{\abs{\gamma}=k,\abs{\delta}=l,\atop \gamma\cap\delta\ni j,
    (\gamma\cup\delta)\cap(\alpha\cup\beta)\neq\emptyset}
    32\rho\\
    &\qquad= \frac{1}{n^2}\sum_{i=1}^n
    \sum_{\abs{\alpha}=k,\abs{\beta}=l,\atop \alpha\cap\beta\ni i}
    \bbbklr{
    \sum_{j\in\alpha\cup\beta}
    \sum_{\abs{\gamma}=k,\abs{\delta}=l,\atop \gamma\cap\delta\ni j}
    32\rho +
    \sum_{j\not\in\alpha\cup\beta}
    \sum_{\abs{\gamma}=k,\abs{\delta}=l,\atop \gamma\cap\delta\ni j,
    (\gamma\cup\delta)\cap(\alpha\cup\beta)\neq\emptyset}
    32\rho}\\
    &\qquad =: A_{k,l} + B_{k,l},
\ee
where the equality is just a split of the sum
over $j$ into the cases whether or not $j\in\alpha\cup\beta$. In the former
case we automatically have
$(\alpha\cup\beta)\cap(\gamma\cup\delta)\neq\emptyset$. It is now not difficult
to see that
\be
  A_{k,l}\leq \frac{32\rho(k+l-1)}{n}{n-1\choose k-1}^2{n-1\choose l-1}^2.
\ee
Noting that, for fixed $j$, $k$, $l$, $\alpha$ and $\beta$,
\bes
    &\bklg{\abs{\gamma}=k,\abs{\delta}=l\,:\,\gamma\cap\delta\ni j,
    (\gamma\cup\delta)\cap(\alpha\cup\beta)\neq\emptyset}\\
    &\qquad = \bklg{\abs{\gamma}=k,\abs{\delta}=l\,:\,\gamma\cap\delta\ni j}\\
    &\qquad\quad\setminus
    \bklg{\abs{\gamma}=k,\abs{\delta}=l\,:\,\gamma\cap\delta\ni j,
    (\gamma\cup\delta)\cap(\alpha\cup\beta)=\emptyset},
\ee
we further have
\bes
  B_{k,l}&\leq\frac{32\rho(n-1)}{n}{n-1\choose k-1}{n-1\choose l-1}\times\\
  &\qquad\times\bbbklg{{n-1\choose k-1}{n-1\choose l-1}
    -{n-k-l+1\choose k-1}{n-k-l+1\choose l-1}},
\ee
where we also used that ${n-\abs{\alpha\cup\beta}\choose
k-1}\geq{n-k-l+1\choose k-1}$.
The following statements are straightforward to prove:
\bgn     
    {n-1\choose k-1}{n\choose k}^{-1} = \frac{k}{n},            \label{18}\\
    {n-k-l+1\choose k-1}{n\choose k}^{-1}\geq
    \frac{k}{n}\bbbklr{\frac{n-2k-l+3}{n}}^k
    \geq \frac{k}{n}\bbbklr{1-\frac{k(2k+l-3)}{n}}.   \label{19}
\ee
Thus, from \eq{18},
\ben
    n^2{n\choose k}^{-2}{n\choose l}^{-2}A_{k,l} 
    \leq \frac{32\rho (k+l-1) k^2 l^2}{n^3}
    \leq \frac{64\rho d^5}{n^3}.
\ee
From \eq{18} and \eq{19},
\be
    n^2{n\choose k}^{-2}{n\choose l}^{-2}B_{k,l} 
    \leq \frac{32\rho k^2 l^2\bklr{k(2k+l-3)+l(k+2l-3)}}{n^3}
    \leq \frac{192\rho d^6}{n^3}.
\ee
Thus, for all $k$ and $l$,
\besn                                                            \label{20}
    \Var\IE^W(W'_k-W_k)(W'_l-W_l) 
    & \leq \Var\IE^{X,X'}(W'_k-W_k)(W'_l-W_l)\\
    & \leq \frac{256\rho d^6}{n^3}.
\ee
Notice further that for any $m=1,\dots,d$,
\bes
    \IE\abs{U'_m-U_m}^3 
    & = \frac{1}{n}\sum_{j=1}^n\IE\bbabs{
    \sum_{\abs{\alpha}=\abs{\beta}=\abs{\gamma}=m
            \atop\alpha\cap\beta\cap\gamma\ni j} 
    \eta_{j,m}(\alpha)\eta_{j,m}(\beta)\eta_{j,m}(\gamma)
    }\\
    & \leq 8\rho^{3/4}{n-1 \choose m-1}^3,
\ee
using \eq{26}; hence, along with \eq{18},
\besn                                                   \label{21}
    \IE\abs{(W_i'-W_i)(W_k'-W_k)(W_l'-W_l)} 
    &\leq \max_{m=i,k,l}\IE\abs{W_m'-W_m}^3\\
    & \leq 
    8\rho^{3/4} n^{3/2}\max_{m=i,k,l}{n \choose m}^{-3}{n-1 \choose m-1}^3\\
    & \leq 8 \rho^{3/4} d^3 n^{-3/2}.
\ee
Applying Theorem~\ref{thm1} with the estimates \eq{15}, \eq{20} and
\eq{21} proves the claim.
\end{proof}

\begin{remark}
Using the operator norm as used by Meckes \cite{Meckes2009} we would be able to
achieve a bound of $n\log(d+1)$ instead of \eq{15}, but using bounds for the
total derivatives of the test functions,  $\sup_{x \in \RR^k} \| D^r h(x)
\|_{op}$, instead of bounds for $|h|_r$. 
\end{remark}

\section{Edge and triangle counts in Bernoulli random graphs} \label{randomgraphs} 

Typical summaries for random graphs are the degree distribution and the number
of triangles, as a proxy for the clustering coefficient in a random graph, which
is the expected ratio of the number of triangles over the number of 2-stars a
randomly chosen vertex is involved with. Conditional uniform graph tests are
based on fixing the degree distribution and randomising over the edges,
conditional on keeping the degree distribution fixed. Our next example shows
that even when fixing only the number of edges, not even the degree
distribution, under a normal asymptotic regime the number of triangles, and the
number of 2-stars, is already asymptotically determined. 
Let $G(n,p)$ denote a Bernoulli random graph on $n$ vertices, with edge
probabilities $p$;  we assume that $n \ge 4$ and that $0<p<1$. Let $I_{i,j} =
I_{j,i} $ be the Bernoulli($p$)-indicator that edge $(i,j)$ is present in the
graph; these indicators are independent. Our interest is in the joint
distribution of the total number of edges, described by 
\be
    T =  \frac{1}{2} \sum_{ i, j  } I_{i,j} = \sum_{i<j} I_{i,j}
\ee
and  the number of triangles, 
\be
    U =  \frac{1}{ 6} \sum_{i,j,k \text{ distinct }
}  I_{i,j} I_{j,k} I_{j,k} = \sum_{i<j<k} I_{i,j} I_{j,k} I_{j,k} .
\ee
Here and
in what follows, ``$ i,j,k \text{ distinct}$'' is short for ``$(i,j,k): i \ne j
\ne k \ne i$''; later we shall also use  ``$ i,j,k, \ell  \text{ distinct}$'',
which is the analogous abbreviation for four indices. 

In view of the embedding method, we also include the auxiliary statistic
related to the number of 2-stars, 
\bes  
    V & := \frac{1}{2}  \sum_{ i,j,k \text{ distinct } }  I_{i,j} I_{j,k} 
    \enskip=   \sum_{  i < j < k }  ( I_{i,j} I_{j,k}  
        + I_{i,j} I_{i,k} + I_{j,k} I_{i,k})
.
\ee 
We note that 
\be 
    \IE T = {n \choose 2} p; \qquad \IE V = 3 {n \choose 3} p^2, \mbox{ and } 
\qquad \IE U =  {n \choose 3} p^3.
\ee
With some calculation, as detailed in
Section~\ref{sec5}, we find that the variances are not all of the same order.
Hence, we re-scale our
variables (c.f.~\cite{Janson1991}), putting 
\be
T_1 = \frac{n-2}{n^2} T, \qquad 
V_1 = \frac{1}{n^2} V, \qquad \mbox{ and } 
U_1 = \frac{1}{n^2} U.
\ee
For these re-scaled variables the covariance matrix $\Sigma_1$ for $W_1=(T_1
-
\IE T_1,V_1 - \IE V_1,U_1 - \IE U_1)$ equals 
\ben                                                         \label{23} 
{\Sigma_1} 
 = 
 3 \frac{(n-2) {n \choose 3}}{n^4}  p (1-p) \times
  \left( \begin{array}{c c  c } 
1  &   2 p  &   p^2
 \\   2 p  &    4 p^2 + \frac{p(1-p)}{n-2 }  & 
 2 p^3 + \frac{p^2(1-p)}{n-2} 
 \\   p^2 &   2 p^3 + \frac{p^2(1-p)}{n-2} 
 &  
p^4 + \frac{p^2(1+p-2p^2)}{3(n-2) } 
\end{array} \right) .
\ee

\begin{remark}
With $n \rightarrow \infty$ we obtain as approximating covariance matrix 
\ben                                                            \label{24} 
{\Sigma_0} 
 = 
\frac{1}{2}  p (1-p) \times 
 \left( \begin{array}{c c  c } 
1  &   2 p  &   p^2
 \\   2 p  &    4 p^2 & 
 2 p^3 
 \\   p^2 &   2 p^3 
 &  
p^4  
\end{array} \right) .
\ee
As also observed in \cite{Janson1991}, this matrix has rank 1. It is not
difficult to see that
the maximal diagonal entry of the inverse $\Sigma^{-1}$ tends to $\infty$ as $n
\rightarrow \infty$, so that a uniform bound on the square root of
$\Sigma_1^{-1}$, as suggested in Remark~\ref{cor2},  will not be useful. 
\end{remark}

Janson and Nowicki \cite{Janson1991} derived a normal limit for $W_1$, but no
bounds on the approximation are given. Using Theorem~\ref{thm1} we obtain
explicit bounds, as follows.

\begin{proposition} \label{25} Let $W_1=(T_1 - \IE T_1, V_1 - \IE V_1, U_1 - \IE
U_1)$ be the centralized count vector of the number of edges, two-stars and
triangles in a Bernoulli$(p)$-random graph. Let $\Sigma_1$ be given as in
\eq{23}.  Then, for every three times differentiable function~$h$,
\be
{ \babs{\IE h(W)-\IE h(\Sigma_1^{1/2}Z)} }
    \leq   
    \frac{\abs{h}_2}{n} \left( \frac{35}{4}  + 9 n^{-1} \right) + \frac{8
\abs{h}_3}{3n}  \left( 1 + n^{-1} + n^{-2} \right).  
\ee
\end{proposition}

While we do not claim that the constants in the bound are sharp, as we
have ${n \choose 2}$ random edges in the model, the order $O(n^{-1})$ of the
bound is as expected. While for simplicity our other bounds are given as expressions which are uniform
in $p$, bounds dependent on $p$ are derived on the way. 
In this example, we were not able to obtain any improvement on the bounds using
the operator bounds \cite{Meckes2009}. 

\begin{proof}
The proof consists of two stages. Firstly we construct an exchangeable pair; it
will turn out that $R=0$ in \eq{1} and hence  $C$ in Theorem~\ref{thm1} will
vanish. In the second stage we bound the terms $A$ and $B$ in
Theorem~\ref{thm1}. 

\proofsection{Construction of an exchangeable pair}

Our vector of interest is now $\bX = (T - \IE T ,V - \IE V ,U - \IE U)$,
re-scaled to $\bX_1 = (T_1 - \IE T_1 ,V_1 - \IE V_1,U_1 - \IE U_1)$. We build an
exchangeable pair by choosing a potential edge $(i,j)$ uniformly at random, and
replacing $I_{i,j}$ by an independent copy $I_{i,j}'$. More formally, pick
$(I,J)$ according to
\be
    \IP[I=i, J=j] = \frac{1}{{n \choose 2}}, \quad 1\leq i<j\leq n.
\ee
 If $I=i$, $J=j$ we replace $ I_{i,j} = I_{j,i} $ by an
independent copy $ I_{i,j}'= I'_{j,i} $ and put 
\ba 
T' &= T - ( I_{I,J} -  I_{I,J}' ),  \\
V'
&= V  - \sum_{ k:  k \ne I,J }  (I_{I,J} - I_{I,J}')(I_{J,k}+ I_{I,k}) , 
\\
U' 
&= U - \sum_{ k:  k \ne I,J }  (I_{I,J} - I_{I,J}')I_{J,k} I_{I,k} .
\ee
Put $\bX' = (T' - \IE T ,V' - \IE V ,U' - \IE U). $ Then $(W,W')$ forms an
exchangeable pair. We re-scale $W'$ as for $W$ to obtain $T_1'$, $V_1'$ and
$U_1'$, so that $(W_1, W_1')$ is also exchangeable. 

\proofsection{Calculation of $\Lambda$}

For the conditional expectations $\IE^{\bX} (W'-W) $,  firstly we have 
\bes
 \IE^{\bX} (T_1'-T_1) &= \frac{2(n-2)}{n^3(n-1)} \sum_{i < j} \IE^{\bX} (
I_{i,j}' - 
I_{i,j}
\vert I=i, J=j) \\
&= \frac{n-2}{n^2}  p -  \frac{2(n-2)}{n^3(n-1)} T 
= - \frac{1}{{n \choose 2} }\left(  T_1  - \IE T_1 \right). 
\ee
Furthermore 
\bes
- \IE^{\bX} (V_1'-V_1)
&\qquad= \frac{1}{n^2 {n \choose 2}} \sum_{i < j} \IE^{\bX} \sum_{ k: k\neq i,j
} 
(I_{i,j} - I_{i,j}') (I_{j,k}+ I_{i,k}) \\
&\qquad= 2 \frac{1}{n^2 {n \choose 2}}  V  - 2 p \frac{1}{ n^2 {n \choose 2} }
(n-2)
T\\
&\qquad= -2 \frac{1}{ {n \choose 2} } (V_1 - \IE V_1)  + 2 p \frac{1}{ {n
\choose 2} } (T_1 - \IE T_1),
\ee
 where the last equality follows from $\IE (V_1' - V_1) =0$. 
Similarly, 
\bes 
- \IE^{\bX} (U_1'-U_1) 
&= -3 \frac{1}{ {n \choose 2} } (U_1 - \IE U_1)  +  p \frac{1}{ {n
\choose 2} } (V_1 - \IE V_1), 
\ee 
Using our re-scaling, \eq{1} is satisfied with $R=0$ and $\Lambda$ given by 
\beas 
\Lambda = 
\frac{1}{ {n \choose 2} } \left( \begin{array}{c c  c } 
 1 & 0 & 0 
 \\- 2 p   & 2 & 0 
 \\ 
0&  - p  & 3
\end{array} \right) .
\enas

\proofsection{Bounding $A$}

The inverse matrix $\Lambda^{-1}$ is easy to calculate; 
for  $ \lambda^{(i)} = \sum_{m=1}^d\abs{(\Lambda^{-1})_{m,i}}$, 
for simplicity we shall apply the uniform bound
\be
    \abs{  \lambda^{(i)} } \leq\frac{3}{2} n^2, \quad i=1, 2, 3. 
\ee 
As the bounding of the conditional variances is somewhat laborious, most of the
work can be 
found in Section~\ref{sec6}. 
The conditional variances involving $T'-T$ can be calculated exactly. As
$I_{i,j}^2 = I_{i,j}$, 
\bes
\IE^W(T'-T)^2 
&= \frac{1}{{n \choose 2}} \sum_{i < j} \IE^W( I_{i,j}' - I_{i,j} )^2\\
&= \frac{1}{{n \choose 2}} \sum_{i < j}\bklg{ p - p  \IE^W I_{i,j}   +
(1-p)\IE^W I_{i,j} }   = p +  (1-2p)  \frac{1}{{n \choose 2}} T, 
\ee
so that with $\Var T$ given in \eq{vart},  
$
\Var ( \IE^W(T'-T)^2) = \frac{1}{{n \choose 2 }} (1-2p)^2 p(1-p) 
$
and 
\be
\Var ( \IE^W(T_1'-T_1)^2) = \frac{(n-2)^4}{n^8 {n \choose 2 }} (1-2p)^2 p(1-p)
< n^{-6}, 
\ee
where we used that $p(1-p) \leq1/4$ for all $p$. 
Thus 
\be 
    \sqrt{\Var ( \IE^W(T_1'-T_1)^2))} < n^{-3}.
\ee

\medskip 
Similarly we obtain 
\be
    \sqrt{ \Var \IE^W(T_1'-T_1)(V_1'-V_1) } < 2 n^{-3}
\ee
and 
\be
    \sqrt{ \Var \IE^W(T_1'-T_1)(U_1'-U_1) } < 2 n^{-3}.
\ee

\medskip 
The conditional variances involving $V$ and $U$ only are more involved; the
calculations 
are available in Section~\ref{sec6}.
We obtain after some calculations that 
$
 \Var ( \IE^W (V'-V)^2 ) \leq 33 n^2, $
and hence 
\be 
 \sqrt{ \Var ( \IE^W (V_1'-V_1)^2) } < 6 n^{-3}.  
\ee
For $ \IE^W (V_1'-V_1)(U_1'-U_1)$, analogous calculations lead to  
\be
\sqrt{ \Var \left( \IE^W (V_1'-V_1)(U_1'-U_1) \right)} <   n^{-3} + 11
n^{-4}. 
\ee
Finally, again using our variance inequalities, 
\be
\sqrt{ \Var \left(  \IE^W (U_1'-U_1)^2\right) } < 5 n^{-3} + 2 n^{-4}. 
\ee
Collecting these bounds we obtain  for $A$ in Theorem~\ref{thm1} that 
\be
A < 35 n^{-1} + 36 n^{-2} .
\ee

\proofsection{Bounding $B$}

We use the generalized H\"{o}lder inequality 
\ben
 \label{26}
\IE \prod_{i=1}^3 \vert X_i \vert \leq\prod_{i=1}^3 \{ \IE \vert
X_i\vert^3\}^{\frac{1}{3}} \leq\max_{i=1,2,3} \IE \vert X_i \vert^3.
\ee 
Again the complete calculations are found in 
Section~\ref{sec7}. To illustrate the calculation, 
\be
\IE \vert T'-T \vert^3 =  \frac{1}{{n \choose 2}} \sum_{i < j} \IE \vert
I_{i,j} - I_{i,j}' \vert^3 
= 2p(1-p) \leq \frac{1}{2} , 
\ee
so that 
\be
\IE \vert T_1'-T_1 \vert^3 = \frac{(n-2)^3}{n^6} 2p(1-p) < \frac{1}{2}
n^{-3}. 
\ee
Similar calculations yield , 
\be
{\IE \vert V_1'-V_1\vert^3  }
< \frac{64}{27} \left(  n^{-3} +  n^{-4} +  n^{-5}    \right) ,
\ee
as well as 
\be
\IE \vert U_1'-U_1  \vert^3  
< \frac{27}{128} \left(  n^{-3} +  n^{-4} +  n^{-5}    \right).
\ee

Thus for $B$ in Theorem~\ref{thm1} we have 
\be
  B < \frac{3}{2} n^2 
  \times 9 \times  \frac{64}{27} \left(  n^{-3} +  n^{-4} +  n^{-5}    \right) 
  =  32 \left(  n^{-1} +  n^{-2} +  n^{-3}    \right) .  
\ee
Collecting the bounds gives the result.
\end{proof} 

\medskip
\begin{remark}
Had we not introduced $V$, 
conditioning would yield 
\bes
- \IE ^{T,U} (U'-U) &= \frac{2}{n(n-1)} \sum_{i < j}  \IE^{T,U}\sum_{
k:k\neq i,j} (  I_{i,j}
I_{j,k} I_{i,k} - I_{i,j}' I_{j,k} I_{i,k}) \\
&= 3 \frac{2}{n(n-1)} U  - p \frac{2}{n(n-1)}  \IE^{T,U} \sum_{ 
i<j,\,k\neq i,j}  I_{j,k} I_{i,k} .
\ee

The expression $   \sum_{i<j,\,k\neq i,j} 
\IE^{T,U }  I_{j,k} I_{j,k}$ would result in a non-linear remainder term $R$ 
in Equation~\eq{1}. The introduction of $V$ not only avoids this remainder term,
indeed 
$R=0$ in \eq{1}, but also
yields a more detailed result. 
This observation that the 2-stars form a useful auxiliary statistic can also be
found in \cite{Janson1991}; there it is related to Hoeffding-type
projections. 
\end{remark} 

\medskip 
Using Proposition~\ref{lem}, we also obtain a normal approximation for
$\Sigma_0$ given in \eq{24}. 
\begin{corollary} 
Under the assumptions of Proposition~\ref{25},  for
every three times differentiable function~$h$,
\bes
    \babs{\IE h(W)-\IE h(\Sigma_0^{1/2}Z)}
   &\leq
    \frac{\abs{h}_2}{2n}  \left(  44 + 21 n^{-1}   + 32 n^{-2} + 4 n^{-3}
\right) \\
    &\quad 
    + \frac{8 \abs{h}_3 }{3n}  \left( 1  
+  n^{-1}  
+  n^{-2} \right)   .   
\ee
\end{corollary}  

\begin{proof}
We employ Proposition~\ref{25} and Proposition~\ref{lem}, with the triangle
inequality. A straightforward calculation shows that 
$
\bbbabs{  \frac{3(n-2) {n \choose 3} }{n^4} - \frac{1}{2} }  
\leq \frac{3}{2} n^{-1} + 2 n^{-3} 
 $
and so 
\bes
&\sum_{i, j=1}^d \abs{\sigma_{i,j} - \sigma^0_{i,j}} \\
&\quad\leq  \left( \frac{3}{2} n^{-1} + 2 n^{-3}  \right) \left\{  1 + 
4p + 6 p^2 + 4 p^3 + p^4 
 \right\} \\
 &\qquad+  \left( \frac{p(1-p)}{n-2}+ 2 \frac{p^2(1-p)}{n-2} 
+   \frac{p^2(1-p)(4-p)}{3(n-2) } \right) \left( \frac{3}{2} n^{-1} + 2 n^{-3} 
+1 \right)
\\
&\quad<  26 n^{-1} +   3 n^{-2} + 32 n^{-3} + 4 n^{-4} 
.
\ee
Here we used the crude bound that $(n-2)^{-1} \leq\frac{3}{2} n^{-1}$. The
corollary follows. 
\end{proof}
 
\subsection{Calculation of the covariance matrix}\label{sec5}

To calculate the covariance matrix $\Sigma$, we put
\be
    \tI_{i,j} = I_{i,j} - p 
\ee
as the centralized edge indicator, 
and similarly we centralize 
\ba
{\tilde{T}}  &=   \sum_{ i <  j  } \tI_{i,j}, \quad 
{\tilde{V}} &=   \frac{1}{ 2} \sum_{i,j,k \text{ distinct } } 
                    \tI_{i,j} \tI_{j,k}
 \mbox{ and } {\tilde{U} } &=   \sum_{i < j < k }  \tI_{i,j}
\tI_{j,k} \tI_{i,k}.
\ee
Then, by independence, all these quantities have mean zero.  
For the variances, the expectation of the product of centralized indicators
vanish unless all the 
centralized indicators involved are raised to an even power. Hence 
\ban \label{27} 
\Var {\tilde{T}}  &= &  {n \choose 2} p(1-p),  \quad
 \Var {\tilde{V} } =   3 {n \choose 3} p^2 (1-p)^2,  
\Var {\tilde{U} } =   {n \choose 3}  p^3 (1-p)^3.
\ee
Moreover, for the same reason, all covariances between the centralised variables
vanish. 
Expressing  $T,V$ and $U$, we have ${\tilde{T}}=T - \IE T$ so that 
\be 
\label{28}
T = {\tilde{T}}  +  \IE T =  {\tilde{T}} + {n \choose 2} p
\ee
and 
\ben \label{vart}
\Var T  =   {n \choose 2} p(1-p) = 3 {n\choose 3} \frac{1}{n-2} p(1-p) .   
\ee
Next, 
$
{\tilde{V}} 
=  V - 2p (n-2) T +  3 p^2 {n \choose 3} $, 
so that 
\be                                           
 \label{29}  
V=   {\tilde{V}} + 2 (n-2) p  {\tilde{T}} 
+  3 {n \choose 3}  p^2 . 
\ee
As $\tilde{V}$ and $\tilde{T}$ are uncorrelated, this gives that 
\ben \label{varv} 
\Var V
= 3 {n \choose 3}  p^2 (1-p) \{ 1-p+ 4(n-2) p\}.
\ee
For $U$, we have 
$
{\tilde{U} }
= U - p V + p^2 (n-2) T - p^3 {n \choose 3}.
$
Using the above expressions \eq{28} and \eq{29} for $T$ and  $V$  we obtain
\be 
U = {\tilde{U} } + p {\tilde{V}} +  p^2(n-2) {\tilde{T}} + p^3 {n \choose 3}. 
\ee
 This gives for the variance 
\ben \label{varu} 
\Var U 
= {n \choose 3} p^3 (1-p) \left\{ (1-p)^2 + 3 p ( 1-p)  + 3 (n-2)  p^2
\right\}  .   
\ee
We can now also calculate the covariances. Again we use that the centralized
variables are uncorrelated
to obtain 
\bes
\Cov(T,V) &= \Cov\left(  {\tilde{T}} , {\tilde{V}} + 2(n-2) p 
{\tilde{T}} \right) =2(n-2) p \Var( {\tilde{T}} ) = 6 {n \choose 3} p^2(1-p).
\ee
 Similarly, we calculate that 
$
 \Cov(T,U) 
 = 3 {n \choose 3} p^3(1-p), 
$
 and   
$
 \Cov(V,U) 
 = 3{n \choose 3} p^3 (1-p) \left( 1 - p + 2  (n-2) p \right)
$.
Re-scaling gives  the covariance matrix~\eq{23}.

\subsection{Calculation of the conditional variances}\label{sec6}

For the conditional variances in the random graph example, the calculations are
somewhat involved.
We repeat the first calculation in more detail before moving on to further
bounds. With \eq{29}, 
\bes
&\IE^W(T'-T)(V'-V)\\
&\qquad= -\frac{1}{{n \choose 2}} \sum_{i < j,\,k \ne i,j} 
\IE^W ( I_{i,j}' - I_{i,j} )^2 ( I_{i,k} + I_{j,k}) \\
  &\qquad= \frac{1}{{n \choose 2}} (-2(n-2) p T - 2(1-2p)V)
\ee
so that
\bes
&\Var \IE^W(T'-T)(V'-V)  = \\
&\qquad\frac{4(n-2) }{{n \choose 2}}  p (1-p)\left\{  (n-2) p^2(3-4p)^2  
+ (1-2p)^2 p(1-p)  \right\} < 4,
\ee
where we used that $p^3(1-p) \leq\frac{27}{256}$ and that $n \ge 4$. Thus 
\be  
    \sqrt{ \Var \IE^W(T'-T)(V'-V) } <2  n^{-3}.
\ee

Similarly, with  \eq{29}, 
\bes
&\IE^W (T'-T)(U'-U)\\
&\qquad= \frac{1}{{n \choose 2}} \sum_{i < j,\,k \ne i,j} \bklg{ p  \IE^W  I_{j,k}
I_{i,k} + (1-2p) \IE^W  I_{i,j}  I_{j,k} I_{i,k}} 
\\
&\qquad= \frac{1}{{n \choose 2}} \bklr{pV + 3(1-2p) U}. 
\ee 
Thus we calculate that 
\bes
\lefteqn{\Var \IE^W (T'-T)(U'-U)}\\
=&
 \frac{n-2 }{{n \choose 2}}   p^3  (1-p) \left(   3(1-2p)^2   (1-p)^2  +  p (1-p) 
 ( 4 - 6 p)^2 +  (n-2) p^2 (5-6p)^2  \right)
\ee
and, using that $p(5-6p) \leq\frac{25}{24}$ and $p^3(1-p) \leq\frac{27}{256}$, we obtain 
\be
    \sqrt{ \Var \IE^W(T'-T)(U'-U) } < n^{-3}.
\ee

For $\Var \IE^W (V'-V)^2$ 
we introduce the notation
\ben                                                       
N_i = \sum_{j: j \ne i} I_{i,j} , \qquad
M_{i,j} = \sum_{k: k  \ne i,j} I_{i,k} I_{k,j}. 
\ee 
Then
\ban
T &= \frac{1}{2} \sum_i N_i, \label{31} \\
V &= \frac{1}{2} \sum_{i \ne j} M_{i,j} 
  =  \frac{1}{2} \sum_{i \ne j} I_{i,j} N_i - T
  = \frac{1}{2} \sum_{i}  N_i^2 - T,
    \label{32} \\
U &= \frac{1}{6} \sum_{i \ne j} I_{i,j} M_{i,j}. \label{33} 
\ee
We have 
\ba
&\IE^W (V'-V)^2\\
&\quad= \frac{1}{{n \choose 2}} \sum_{i < j}  \IE^W   (I_{i,j} - I_{i,j}')^2
\left( 
N_j + N_i - 2 I_{i,j} \right)^2 \\
&\quad=  \frac{1}{2 {n \choose 2}} \sum_{i \ne j}  \left\{  p \IE^W  \left( 
N_j + N_i - 2 I_{i,j} \right)^2   + (1-2 p) \IE^W I_{i,j} \left( 
N_j + N_i - 2 I_{i,j}  \right)^2 \right\} \\
 \\
&\quad=  \frac{1}{2 {n \choose 2}}   \bbbbklgl p \IE^W  \left( 
4 (n-2)(V+T)  - 8 T   + 8 T^2    - 16 V     )  \right)  \\
&\qquad\qquad\quad + (1-2 p)  \IE^W\bbbklr{ 2  \sum_{i \ne j} I_{i,j} N_i^2 
- 8T   + 2 \sum_{i \ne j}   I_{i,j} N_i N_j  - 16 V ) 
 } \bbbbklgr ,
\ee 
where we used \eq{31} and \eq{32} for the last equation. 
To simplify this expression,  note that  $ \sum_i N_i^2 = 2 T  + 2 V$, 
and $ \sum_{i \ne j}  N_i N_j = 4T^2 - 2T - 2V$
as well as 
$$
\sum_{i \ne j} I_{i,j}  N_i^2 
=  \sum_{  i, j,  k , \ell \text{ distinct } }  I_{i,j}   I_{i,k}
I_{i,\ell} +
6 V + 2T , 
$$
and 
$$\sum_{i \ne j} I_{i,j}  N_i N_j
= \sum_{  i, j,  k, \ell  \text{ distinct }  }  I_{i,j}  I_{i,k}  I_{j,
\ell} + 4 V + 6 U 
+ 2T , 
$$ 
so that 
\bes
\IE^W (V'-V)^2 
&=  \frac{1}{ {n \choose 2}}   \bbbklgl
2 p  (n -4) T + 2 V(np - 10p + 2) + 6(1-2p) U + 4 p T^2 \\
&\qquad\qquad\qquad\qquad\qquad + (1-2 p)   \sum_{  i, j,  k , \ell \atop \text{
distinct }
} \IE^W I_{i,j} I_{i,k} ( I_{i,\ell}  + I_{j, \ell}  )
\bbbklgr .
\ee
With the notation $ {\tilde T} $ for the centralized variable, we have that
\bes
&\Var \IE^W (V'-V)^2 \\
&\quad\leq 5 \frac{1}{ {n \choose 2}^2}
\bbbbklgl
 p^2  (2n -8 + 4pn^2 - 4 p n )^2  \Var (T)  + 4 (np - 10p + 2)^2 \Var(V) \\
&\quad\qquad\qquad\qquad + 36(1-2p)^2 \Var(U)  + 16 p^2 \Var( {\tilde T}^2  ) \\
&\quad\qquad\qquad\qquad + (1-2 p)^2 \Var\bbklr{    \sum_{  i, j,  k , \ell
\atop \text{
distinct } } \IE^W I_{i,j}   I_{i,k} ( I_{i,\ell}  + I_{j, \ell}  )}
 \bbbbklgr ,
\ee
where we used that in  general $\Var \sum_{i=1}^k X_i  \leq k \sum_{i=1}^k \Var
X_i$ and \eq{27}. 
Here, the variances for $T, V$ and $U$ are given in \eq{vart}, \eq{varv}, and  \eq{varu}. To simplify the
expression, we use that $ p^3(1-p) \leq27/256$ to  bound
\be 
 p^2  (2n -8 + 4pn^2 - 4 p n )^2  \Var (T) 
\leq  \frac{27}{64 } {n \choose 2} n^2(n+ 2)^2 .
\ee 
Similarly, we bound with $p^2(1-p) \leq4/27 $ and $n \ge 4$ 
\be 
{  4 (np - 10p + 2)^2 \Var(V) } 
\leq   \frac{16}{27} n^3(n-1)(n-2) (n+1),  
\ee 
and 
\be
{36(1-2p)^2 \Var(U) }
\leq  \frac{81}{256} n(n-1)(n-2) (3n+2) . 
\ee
We note that $\IE {\tilde{I}}_{i,j}
{\tilde{I}}_{u,v}{\tilde{I}}_{s,t}{\tilde{I}}_{k,\ell} =0$ unless either all
pairs of indices are
the same, or the product is made up of two distinct index pairs only. Hence 
\bes 
\Var {\tilde T}^2 &= \sum_{i<j} \sum_{u<v} \sum_{s<t} \sum_{k < \ell} \IE
{\tilde{I}}_{i,j} {\tilde{I}}_{u,v}{\tilde{I}}_{s,t}{\tilde{I}}_{k,\ell} \\
&= {n \choose 2} p(1-p)  \left\{ 3  \left( {n \choose 2} -1 \right) p (1-p) + 
(1-p)^3 +  p^3 \right\} \\
&< n^2 {n \choose 2} p(1-p), 
\ee
giving 
$$
 16 p^2 \Var {\tilde T}^2 \leq\frac{27}{32} n^3(n-1)  . 
$$ 
For the last  variance term,  we use that  conditional variances can be bounded
by  unconditional
variances, giving
\bes  \label{34} 
&\Var  \sum_{i \ne j} \sum_{k:k \ne i,j}  \sum_{\ell:\ell \ne i,j, k }
\IE^W I_{i,j}   I_{i,k} (  I_{i,\ell} 
+ I_{j, \ell} )   \nonumber  \\
&\qquad\leq 
\sum_{  i, j,  k, \ell \atop \text{ distinct } }  \Var  I_{i,j}  
I_{i,k} (
 I_{i,\ell} 
+ I_{j, \ell} )   \nonumber \\
&\qquad\qquad+ \sum_{  i, j,  k, \ell  \text{ distinct } }   \sum_{  r,s, t, u  
\text{ distinct } }   {\bf 1}( (i,j,k,\ell) \ne (r,s,t,u))  \nonumber \\
&\qquad\qquad \times 
{\bf 1}( \vert \{ i,j,k, \ell\} \cap \{ r,s,t,u \} \vert \ge 2) \nonumber \\
&\qquad\qquad\times \Cov (  I_{i,j}   I_{i,k} (  I_{i,\ell} 
+ I_{j, \ell} ) ,  I_{r,s}   I_{r,t} ( I_{r,u}  + I_{s,u}) )  )  \nonumber\\
&\qquad\leq 2 {n \choose 4} \left(  p^3(1-p^3) +  4 {4 \choose 2}  {n
\choose 2} p^2(1 - p^4)  \right) 
< 3 n^2  {n \choose 4}  \nonumber. 
\ee 
Here we used  the independence of the edge indicators. For the last bound we
employed that $p^3(1-p^3) \leq1/4$, 
that $p^2(1-p^4) \leq( \sqrt{3}-1) /3$, and that $n \ge 4$. Collecting the
variances and using that $n \ge 4$, 
\bes 
&\Var (\IE^W (V'-V)^2)\\
&\quad\leq5 \frac{1}{ {n \choose 2}^2}
\left\{
 \frac{27}{64 } {n \choose 2} n^2(n+ 2)^2  + \frac{16}{27} n^3(n-1)(n-2) (n+1)
\right.\\
&\quad\kern4.5em\left.  + \frac{81}{256} n(n-1)(n-2) (3n+2)   + \frac{27}{32}
n^3(n-1)  + 3 n^2  {n \choose 4}  
 \right\} 
\ee 
This gives that 
\be 
 \sqrt{ \Var ( \IE^W (V_1'-V_1)^2) } < 6 n^{-3} .  
\ee

For $ \IE^W (V'-V)(U'-U)$, we have, 
\bes
&\IE^W (V'-V)(U'-U)\\
&\qquad= \frac{1}{{n \choose 2}} \bbbklr{
 p  \sum_{i \ne j} \IE^W N_i M_{i,j} - 6(1-p) U   +   (1-2p) \sum_{i \ne  j}  
\IE^W I_{i,j} N_i M_{i,j} 
 }  .
\ee
where we used  \eq{33}.
Now
\bes
\sum_{i \ne  j} N_i M_{i,j}
&= 2 V   + 6 U  + 
\sum_{ i,j,k ,\ell \atop  \text{ distinct } }  I_{i,k} I_{k,j} I_{i,\ell}  .
\ee
Similarly, 
\be
{ \sum_{i \ne  j}I_{i,j} N_i M_{i,j}}
= 12  U +  \sum_{ i,j,k  , \ell \atop \text{ distinct } } I_{i,j} I_{i,k}  
I_{i,\ell} I_{\ell, j},
\ee 
so that 
\bes
\IE^W (V'-V)(U'-U)
&= \frac{1}{{n \choose 2}} \bbbklrl
 2 p V + 6(1-2 p) U  + p \sum_{ i,j,k  , \ell \atop \text{ distinct } }  I_{i,k}
I_{k,j} I_{i,\ell} \\
&\qquad\qquad\qquad\qquad\qquad+ (1-2p) \sum_{ i,j,k  , \ell \atop
\text{ distinct } } I_{i,j} I_{i,k}  I_{i,\ell} I_{\ell, j}
 \bbbklrr .
\ee
Furthermore, as before, 
\bes
&\Var  \sum_{ i,j,k  , \ell  \text{ distinct } }  \IE^W  I_{i,k}
I_{k,j} I_{i,\ell} \\
 &\qquad\leq   {n \choose 4} \left(  p^3(1-p^3) + 6 {n \choose 2} p^2(1 - p^4) 
\right) < {n \choose 4} n^2.
\ee
Similarly as for \eq{34}, 
\bes
\Var  \sum_{ i,j,k  , \ell \atop  \text{ distinct } }  \IE^W  I_{i,j}   I_{i,k} 
I_{i,\ell}  I_{j,\ell} 
&\leq \Var  \sum_{ i,j,k  , \ell \atop \text{ distinct } }   I_{i,j}   I_{i,k} 
I_{i,\ell}  I_{j,\ell}  \\
&\leq  {n \choose 4} \left(  p^4(1-p^4) + 6 {n \choose 2} p^2(1 - p^6)  \right)
\\
&< {n \choose 4} \left(  \frac{1}{256}  + \frac{1}{16}  {n \choose 2} \right) . 
\ee 
As $p<1$, we obtain that 
\bes
&\Var \IE^W (V'-V)(U'-U) \\
&\qquad< 4 \frac{1}{{n \choose 2}^2}  \left\{ 12 \frac{27}{256}  {n
\choose 3} \bklr{ 16 (n-2)  + 1 ) + 9 n  + 9}  \right. \\
&\kern10em\left. +  {n \choose 4} n^2   +  {n \choose 4} \left( 
\frac{1}{256} 
+ \frac{1}{16}  {n \choose 2} \right) \right\} < n^2 + 108
\ee
so that 
\be
\sqrt{ \Var \left( \IE^W (V_1'-V_1)(U_1'-U_1) \right)} <   n^{-3} + 11
n^{-4}. 
\ee

\medskip
Finally, 
\be
\IE^W (U'-U)^2
=  \frac{1}{2{n \choose 2}} \sum_{i \ne j} \bklr{ p \IE^W M_{i,j}^2 + (1-2p)  
\IE^W I_{i,j} M_{i,j}^2}. 
\ee 
We have that 
\be
M_{i,j}^2 = \sum_{k: k \ne i,j} \sum_{\ell: \ell  \ne i,j} I_{i,k}I_{k,j}
I_{i,\ell}I_{\ell,j}
= M_{i,j} +  \sum_{k: k \ne i,j} \sum_{\ell: \ell \ne i,j, k} I_{i,k}I_{k,j}
I_{i,\ell}I_{\ell,j},
\ee 
and 
\bes 
 I_{i,j} M_{i,j}^2 =  I_{i,j} M_{i,j} +  \sum_{k:k \ne i,j} \sum_{\ell:\ell
\ne i,j, k}  I_{i,j} I_{i,k}I_{k,j} I_{i,\ell}I_{\ell,j},
\ee 
so that 
\bes
\IE^W (U'-U)^2
&=  \frac{1}{2{n \choose 2}} \bbbklgl 2 p V  + 6 (1-2p) U 
+ p  \sum_{ i,j,k  , \ell \atop \text{ distinct } }  \IE^W  I_{i,k}I_{k,j}
I_{i,\ell}I_{\ell,j} \\
&\kern8em + (1-2p)  \sum_{ i,j,k,\ell \atop \text{ distinct } }    \IE^W 
I_{i,j}
I_{i,k}I_{k,j} I_{i,\ell}I_{\ell,j}
 \bbbklgr. 
\ee
As for \eq{34}, we obtain  
\be
\Var  \sum_{ i,j,k, \ell \atop \text{ distinct } }   \IE^W   I_{i,k} I_{k,j} 
I_{i,\ell}  I_{j,\ell}  
\leq{n \choose 4} \left(  p^4(1-p^4) + 6 {n \choose 2} p^2(1 - p^6)  \right)
\ee 
and 
\be
\Var  \sum_{ i,j,k  , \ell \atop \text{ distinct } } \IE^W   I_{i,j} I_{i,k}
I_{k,j}  I_{i,\ell}  I_{j,\ell}  
\leq {n \choose 4} \left(  p^5(1-p^5) + 6 {n \choose 2} p^2(1 - p^8) 
\right) . 
\ee
Again using our variance inequalities, we thus obtain that 
\bes
&\Var \left(  \IE^W (U'-U)^2\right) \\ 
&\qquad\leq  \frac{1}{{n \choose 2}^2} \bbbbklgl
 3 {n \choose 3} p^3(1-p) \bbbklrl 4 p ( 4(n-2) p + 1 - p ) \\
&\kern8em  + 36(1-2p)^2( 
 (n-2)p^2  + \frac{1}{3} (4 - 5 p + p^2) ) \bbbklrr \\
&\kern6em
{}+ p^2  {n \choose 4} \left(  p^4(1-p^4) + 6 {n \choose 2} p^2(1 - p^6) 
\right)
\\
&\kern6em{} + (1-2p)^2  {n \choose 4} \left(  p^5(1-p^5) + 6 {n \choose 2}
p^2(1 - p^8)  \right) 
 \bbbbklgr \\
&\qquad\leq   22 + 2 n^2  ,
\ee
so that 
\be
    \sqrt{ \Var \left(  \IE^W (U'-U)^2\right) } <   5 n^{-3} + 2 n^{-4}. 
\ee

\subsection{Calculation of the  third moments}
\label{sec7}

Firstly, $
    \IE \vert T'-T \vert^3 
        =  \frac{1}{{n \choose 2}} \sum_{i < j} \IE \vert I_{i,j} - I_{i,j}'
    \vert^3 = 2p(1-p) < \frac{1}{2} , $
so that 
\be
\IE \vert T_1'-T_1 \vert^3 = \frac{(n-2)^3}{n^6} 2p(1-p) < \frac{1}{2} n^{-3}. 
\ee
Similarly, 
\bes
&\IE \vert V'-V\vert^3 \\
&\qquad= \frac{1}{{n \choose 2}} \sum_{i < j} \IE \vert I_{i,j} - I_{i,j}'
\vert^3 \sum_{k, \ell, s : k, \ell, s  \ne i,j} 
(I_{j,k} + I_{i,k} ) (I_{j,\ell}  +I_{i,\ell} ) (I_{j,s} + I_{i,s} ) \\
&\qquad=  2p(1-p) (n-2)\times\\
&\kern4em\times\left(  8 p^2 + 2 p (1-p)   + 2(n-3) (2p^2 + 2 p^3) +
8(n-3)(n-4) p^3 \right) ,
\ee 
so that 
\be
{\IE \vert V_1'-V_1 \vert^3}
<  \frac{64}{27} \left(  n^{-3} +  n^{-4} +  n^{-5}    \right) . 
\ee 
Lastly, 
\bes
\IE \vert U'-U\vert^3 
&=  \frac{1}{{n \choose 2}} \sum_{i < j} \IE \vert I_{i,j} - I_{i,j}' \vert^3
\sum_{k : k \ne i,j} \sum_{\ell : \ell \ne i,j} 
\sum_{s : s \ne i,j} 
I_{j,k} I_{i,k} I_{j,\ell}  I_{i,\ell} I_{j,s}  I_{i,s}  \\
&=  2p(1-p) (n-2) \left( p^2 + (n-3) p^4 + (n-3)(n-4) p^6  \right) ,
\ee
so that 
\be
\IE \vert U_1'-U_1 \vert^3 
< \frac{54}{256} \left(  n^{-3} +  n^{-4} +  n^{-5}    \right) .
\ee




\begin{thebibliography}{10}

\bibitem{Bolthausen1993}
{\sc Bolthausen, E. and G{\"o}tze, F.} (1993).
\newblock The rate of convergence for multivariate sampling statistics.
\newblock {\em Ann. Statist.\/} {\bf 21,} 1692--1710.

\bibitem{Chatterjee2007}
{\sc Chatterjee, S.} (2007).
\newblock Stein's method for concentration inequalities.
\newblock {\em Probab. Theory Related Fields\/} {\bf 138,} 305--321.

\bibitem{Janson1991}
{\sc Janson, S. and Nowicki, K.} (1991).
\newblock The asymptotic distributions of generalized {$U$}-statistics with
  applications to random graphs.
\newblock {\em Probab. Theory Related Fields\/} {\bf 90,} 341--375.

\bibitem{Lee1990}
{\sc Lee, A.~J.} (1990).
\newblock {\em {$U$}-statistics} vol.~110 of {\em Statistics: Textbooks and
  Monographs}.
\newblock Marcel Dekker Inc., New York.
\newblock Theory and practice.

\bibitem{Meckes2009}
{\sc Meckes, E.~S.} (2009).
\newblock On {S}tein's method for multivariate normal approximation.
\newblock {\em to appear in High Dimensional Probability V\/}.

\bibitem{Nourdin2009b}
{\sc Nourdin, I., Peccati, G. and Reinert, G.}
\newblock Stein's method and stochastic analysis of {R}ademacher functionals.
\newblock arXiv:0810.2890v3 (2009).

\bibitem{Raic2004}
{\sc Rai{\v{c}}, M.} (2004).
\newblock A multivariate {CLT} for decomposable random vectors with finite
  second moments.
\newblock {\em J. Theoret. Probab.\/} {\bf 17,} 573--603.

\bibitem{Reinert2009}
{\sc Reinert, G. and R\"{o}llin, A.} (2009).
\newblock Multivariate normal approximation with Stein's method of exchangeable
  pairs under a general linearity condition.
\newblock {\em Ann. Probab.\/} {\bf 37,} 2150--2173.

\bibitem{Rinott1996}
{\sc Rinott, Y. and Rotar, V.} (1996).
\newblock A multivariate {CLT} for local dependence with {$n\sp {-1/2}\log n$}
  rate and applications to multivariate graph related statistics.
\newblock {\em J. Multivariate Anal.\/} {\bf 56,} 333--350.

\bibitem{Rinott1997}
{\sc Rinott, Y. and Rotar, V.} (1997).
\newblock On coupling constructions and rates in the {CLT} for dependent
  summands with applications to the antivoter model and weighted
{$U$}-statistics.
\newblock {\em Ann. Appl. Probab.\/} {\bf 7,} 1080--1105.

\bibitem{Stein1972}
{\sc Stein, C.} (1972).
\newblock A bound for the error in the normal approximation to the distribution
  of a sum of dependent random variables.
\newblock In {\em Proceedings of the Sixth Berkeley Symposium on Mathematical
  Statistics and Probability (Univ. California, Berkeley, Calif., 1970/1971),
  Vol. II: Probability theory}.
\newblock Univ. California Press, Berkeley, Calif.
\newblock pp.~583--602.

\bibitem{Stein1986}
{\sc Stein, C.} (1986).
\newblock {\em Approximate computation of expectations}.
\newblock Institute of Mathematical Statistics Lecture Notes---Monograph
  Series, 7. Institute of Mathematical Statistics, Hayward, CA.

\end{thebibliography}

\end{document}